\documentclass[12pt]{amsart}

\usepackage{amsmath,amsfonts,amsthm,amssymb}
\usepackage{mathtools}
\usepackage{geometry}
\usepackage{setspace}
\usepackage{hyperref}
\hypersetup{hidelinks}

\hypersetup{
  colorlinks=true,
  linkcolor=blue,
  citecolor=blue,
  urlcolor=blue
}

\newtheorem{theorem}{Theorem}[section]
\newtheorem{lemma}[theorem]{Lemma}
\newtheorem{proposition}[theorem]{Proposition}
\newtheorem{corollary}[theorem]{Corollary}
\theoremstyle{remark}
\newtheorem{remark}[theorem]{Remark}

\usepackage[T1]{fontenc}
\usepackage{lmodern}
\usepackage{microtype}
\usepackage{amsmath,amssymb,amsthm,mathtools}
\usepackage{enumitem}
\newcommand{\Area}{\operatorname{Area}}
\newcommand{\Ric}{\operatorname{Ric}}
\newcommand{\Hess}{\operatorname{Hess}}

\theoremstyle{plain}

\theoremstyle{definition}

\title[The second eigenvalue in $\mathbb{RP}^3$]
{A strict second-eigenvalue estimate for the scalar stability operator on surfaces in $\mathbb{RP}^{3}$}

\author{Márcio Batista and Abraão Mendes}
\address{CPMAT - Instituto de Matemática, Universidade Federal de Alagoas, Maceió, AL, 57072-970, Brazil}
\email{mhbs@mat.ufal.br}
\email{abraao.mendes@im.ufal.br }


\numberwithin{equation}{section}

\usepackage{cancel}
\usepackage[normalem]{ulem}

\begin{document}

\begin{abstract}
Let $\varphi:\Sigma^2\looparrowright\mathbb{RP}^3$ be a closed immersed surface, with no orientability or, equivalently, two-sidedness assumptions. We establish the sharp quantitative estimate
$$
\lambda_2\big(\Delta+|\sigma|^2+2\big)\leq2-\frac{2}{\Area(\Sigma)}\int_\Sigma H^2d\Sigma+\frac{4\pi\chi(\Sigma)}{\Area(\Sigma)}.
$$
Our main contribution is the analysis of the borderline case, which rules out equality when $\chi(\Sigma)\leq0$. More precisely,
$$
\lambda_2\big(\Delta+|\sigma|^2+2\big)<2\quad\text{whenever}\quad\chi(\Sigma)\leq0.
$$
The proof combines the canonical Veronese embedding of $\mathbb{RP}^3$ into $\mathbb{S}^8$, the conformal test function method, an Obata-type rigidity argument, and the classification of closed flat minimal surfaces in $\mathbb{RP}^3$.
\end{abstract}

\maketitle

\section{Introduction}\label{sec:introduction}

The spectrum of the Jacobi operator plays a central role in the study of the stability and rigidity of immersed submanifolds. For a two-sided immersed surface $\varphi:\Sigma^2\looparrowright(M^3,\overline g)$, with globally defined unit normal $N$, the Jacobi operator is 
$$
J=\Delta+|\sigma|^2+\Ric_{\overline g}(N,N),
$$
where $\Delta=\operatorname{div}\nabla$. Its negative eigenvalues determine the Morse index when the immersion is minimal and, more generally, describe the unstable directions for the second variation of area under the appropriate class of variations.

The use of conformal test functions to estimate eigenvalues goes back to Hersch's center-of-mass argument \cite{Hersch} and the conformal-volume method introduced by Li and Yau \cite{LiYau}. These ideas were subsequently adapted to Schrödinger operators and geometric stability problems. In particular, El Soufi and Ilias \cite{ElSoufiIlias2000} obtained general upper bounds for the second eigenvalue of Schrödinger operators in terms of conformal invariants and established sharp estimates for operators involving the squared norm of the second fundamental form in simply connected space forms. Their work extended (to arbitrary codimension and different ambient manifolds) earlier results of Harrell and Loss \cite{HarrellLoss} concerning the Laplacian on surfaces in $\mathbb{R}^3$ penalized by mean curvature.

For surfaces immersed in the unit three-sphere, a sharper picture emerges when the topology is prescribed. Urbano's index theorem \cite{Urbano} identifies the totally geodesic sphere and the Clifford torus as the only closed orientable minimal surfaces in $\mathbb{S}^3$ with Morse index at most five. More directly related to the present work, the second-named author~\cite{Mendes2019} proved that the Clifford torus maximizes the second eigenvalue of the Jacobi operator among all closed orientable immersed surfaces of positive genus in $\mathbb{S}^3$. More precisely, under the sign convention in which the eigenvalues of the Jacobi operator are defined by $Ju+\lambda u=0$, where $u\not\equiv0$, one has $\lambda_2(J)\le-2$, with equality if and only if the surface is congruent to the Clifford torus.

The projective case presents additional features: the ambient space is no longer simply connected, and closed embedded surfaces in $\mathbb{RP}^3$ need not be orientable (equivalently, they may be one-sided). For minimal hypersurfaces, stability and low-index questions in real projective spaces have been investigated by do Carmo, Ritoré, and Ros \cite{doCarmoRitoreRos}. Their results concern the geometric second-variation operator acting on sections of the normal line bundle. By contrast, our purpose is to study the second eigenvalue of the globally defined scalar Schrödinger operator $J=\Delta+|\sigma|^2+2$ acting on functions, for arbitrary closed immersed surfaces in $\mathbb{RP}^3$, without orientability or minimality assumptions.

Henceforth, let $\varphi:\Sigma^2\looparrowright\mathbb{RP}^3$ be an immersion of a closed connected surface, where $\mathbb{RP}^3$ is endowed with its standard metric $g_{\mathbb{RP}^3}$ of sectional curvature one. Locally, choose a unit normal vector field $N$, and let $\sigma$, $A$, and $H=\frac{1}{2}\operatorname{tr}_gA$ denote the second fundamental form, the shape operator, and the mean curvature of $\varphi$, respectively, while $K$ denotes the Gaussian curvature of the induced metric $g=\varphi^*g_{\mathbb{RP}^3}$.

When the immersion is two-sided, its stability operator is 
$$
J=\Delta+|\sigma|^2+\Ric(N,N)=\Delta+|\sigma|^2+2.
$$
If the immersion is one-sided, the geometric second-variation operator acts naturally on sections of the normal line bundle. Nevertheless, the potential $|\sigma|^2+\Ric(N,N)=|\sigma|^2+2$ is globally defined: replacing $N$ by $-N$ changes $A$ into $-A$ but leaves $|\sigma|^2$ unchanged. Thus $J$ defines a scalar Schrödinger operator on $C^\infty(\Sigma)$, independently of the triviality of the normal bundle. This is the analytic point of view adopted throughout the paper.

We denote by
$$
\lambda_1(J)<\lambda_2(J)\le\lambda_3(J)\le\cdots
$$
the eigenvalues of $J$, repeated according to their multiplicities. Our main result provides a quantitative upper bound for $\lambda_2(J)$ in terms of the Willmore energy and the topology of the surface.

\begin{theorem}\label{thm:main}
Let $\varphi:\Sigma^2\looparrowright\mathbb{RP}^3$ be a closed connected immersed surface. For the scalar operator $J=\Delta+|\sigma|^2+2$, we have the sharp estimate
\begin{equation}\label{eq:quantitative-main}
\lambda_2(J)\leq 2-\frac{2}{\Area(\Sigma,g)}\int_\Sigma H^2d\Sigma+\frac{4\pi\chi(\Sigma)}{\Area(\Sigma,g)}.
\end{equation}
Moreover, if $\chi(\Sigma)\leq0$, then $\lambda_2(J)<2$.
\end{theorem}

The estimate \eqref{eq:quantitative-main} is obtained by composing $\varphi$ with the canonical Veronese embedding $\psi:\mathbb{RP}^3\hookrightarrow\mathbb{S}^8$ and applying a Li--Yau center-of-mass normalization to the resulting spherical immersion. The coordinate functions of the balanced immersion serve as admissible test functions for $\lambda_2(J)$. The special extrinsic geometry of the Veronese embedding, combined with the Gauss equation and Gauss--Bonnet theorem, then produces the explicit right-hand side of \eqref{eq:quantitative-main}.

A substantial part of the proof is devoted to the strictness statement $\lambda_2(J)<2$. Indeed, when $\chi(\Sigma)<0$, strictness follows immediately from \eqref{eq:quantitative-main}, whereas the borderline case $\chi(\Sigma)=0$ requires a detailed analysis of the equality conditions. Equality forces the original immersion to be minimal in $\mathbb{RP}^3$ and the balanced Veronese transform to be minimal in $\mathbb{S}^8$. From the transformation law of the mean-curvature vector under Möbius transformations, we obtain a function $f$ satisfying the Obata-type equation
$$
\widehat\Hess f=-\frac{9}{8}f\mskip1mu\widehat g.
$$
The nonconstant case is excluded by Obata's rigidity theorem
\cite{Obata}. In the constant case, the equality identities and the
Gauss--Bonnet theorem force the original surface to be flat and minimal.
Lawson's local classification \cite[Corollary~3]{Lawson1969}
then shows that the immersion is a finite covering of the projective
Clifford torus. A direct computation of the scalar Jacobi operator's spectrum yields $\lambda_2(J)\leq0$, which rules out the equality case $\lambda_2(J)=2$ in Theorem~\ref{thm:main}.

\begin{remark}\label{rem:analytic}
For a one-sided immersion, the spectrum considered here is the spectrum of the globally defined \emph{scalar} operator $J$ acting on functions. It need not coincide with the geometric stability spectrum on sections of the normal line bundle. In particular, no passage to the normal double cover is required in our argument.
\end{remark}

{\bf The paper is organized as follows}. In Section~\ref{sec:preliminaries}, we recall the canonical Veronese embedding of $\mathbb{RP}^3$ into $\mathbb{S}^8$, its relevant extrinsic properties, and the conformal center-of-mass construction. In Section~\ref{sec:eigenvalue-estimate}, we use the balanced Veronese coordinates as test functions and prove the quantitative estimate \eqref{eq:quantitative-main}. Section~\ref{sec:equality-case} is devoted to the analysis of the equality conditions and the derivation of the aforementioned Obata-type equation. In Section~\ref{sec:flat-minimal}, we classify the resulting closed flat minimal immersions and compute the corresponding scalar Jacobi spectrum, thereby completing the proof of the strict inequality. Finally, in Section~\ref{sec:final-comments}, we discuss the sharpness of \eqref{eq:quantitative-main}.

\section{Preliminaries: the Veronese embedding and conformal balancing}\label{sec:preliminaries}

\subsection{The Veronese embedding} 

In this section, we recall the geometric properties of the canonical Veronese embedding that will be used throughout the paper. 

We identify $\mathbb{RP}^3$ with the space of unoriented lines in $\mathbb{R}^4$ and denote by $\operatorname{Sym}_0(4)$ the $9$-dimensional Euclidean space of trace-free symmetric $4\times4$ matrices, endowed with the Frobenius inner product
$$
\langle A,B\rangle=\operatorname{tr}(A^{\mathsf T}B)=\operatorname{tr}(AB).
$$

The canonical Veronese embedding is defined by
$$
\psi:\mathbb{RP}^3\hookrightarrow\mathbb{S}^8\subset\operatorname{Sym}_0(4),\qquad\psi([x])=\frac{2}{\sqrt{3}}\left(xx^{\mathsf T}-\frac{1}{4}I_4\right),
$$
where $x\in\mathbb{S}^3\subset\mathbb{R}^4$. Observe that $xx^{\mathsf T}$ is independent of the choice of representative of $[x]$. Moreover, direct computations show that
$$
\left|xx^{\mathsf T}-\frac{1}{4}I_4\right|^2=\operatorname{tr}\left(xx^{\mathsf T}-\frac{1}{4}I_4\right)^2=\frac{3}{4},
$$
and hence $\psi$ is well defined. 

If $v\in T_{[x]}\mathbb{RP}^3\simeq T_x\mathbb{S}^3=x^\perp$, then 
$$
d\psi_{[x]}(v)=\frac{2}{\sqrt{3}}\left(vx^{\mathsf T}+xv^{\mathsf T}\right).
$$
Since $v\perp x$, we obtain
$$
\left|vx^{\mathsf T}+xv^{\mathsf T}\right|^2=2|v|^2,
$$
and therefore
$$
\psi^*g_{\mathbb{S}^8}=\frac{8}{3}g_{\mathbb{RP}^3}.
$$

Straightforward calculations show that $\psi([x])=\psi([y])$ if and only if $x=\pm y$, that is, $[x]=[y]$. Thus $\psi$ is an isometric embedding of $\left(\mathbb{RP}^3,\frac{8}{3}g_{\mathbb{RP}^3}\right)$ into the unit sphere $\mathbb{S}^8$. This is the standard first-eigenfunction minimal immersion of $\mathbb{RP}^3$ into $\mathbb{S}^8$; see \cite{doCarmoWallach1971,Takahashi1966,Wallach1972}.

Let $B^\psi$ denote the second fundamental form of $\psi$. Let $v\in x^\perp$ be a unit vector with respect to $g_{\mathbb{RP}^3}$, and consider the geodesic $\gamma(t)=[\cos(t)x+\sin(t)v]$ in $\mathbb{RP}^3$. Reparametrizing $\gamma(t)$ so that the corresponding tangent vector $V=\gamma'(0)$ has unit length with respect to~$\frac{8}{3}g_{\mathbb{RP}^3}$,\linebreak a direct differentiation gives
$$
B^\psi(V,V)=\frac{\sqrt{3}}{2}\left(vv^{\mathsf T}-xx^{\mathsf T}\right)+\frac{2}{\sqrt{3}}\left(xx^{\mathsf T}-\frac{1}{4}I_4\right).
$$

Choosing an orthonormal basis of $\mathbb{R}^4$ whose first two elements are $x$ and $v$, respectively, we can see that the preceding matrix has eigenvalues
$$
0,\qquad\frac{1}{\sqrt{3}},\qquad-\frac{1}{2\sqrt{3}},\qquad-\frac{1}{2\sqrt{3}}.
$$
Hence,
\begin{equation}\label{eq:veronese-normal-curvature}
|B^\psi(V,V)|^2=\frac{1}{3}+2\left(\frac{1}{12}\right)=\frac{1}{2}.
\end{equation}
In particular, this quantity is independent of the unit vector $V$.

Now let $\varphi:\Sigma^2\looparrowright\mathbb{RP}^3$
be an immersed surface, let $g=\varphi^*g_{\mathbb{RP}^3}$ be the metric on $\Sigma$ induced from the
standard metric of sectional curvature one on $\mathbb{RP}^3$, and set $X=\psi\circ\varphi:\Sigma^2\looparrowright\mathbb{S}^8$. The metric on $\Sigma$ induced by $X$ is $\widehat g=X^*g_{\mathbb{S}^8}=\frac{8}{3}g$. 

Let $N$ be a local $g_{\mathbb{RP}^3}$-unit normal to $\varphi$, and put $\widehat N=\sqrt{\frac{3}{8}}N$. Then $\widehat N$ is unit with respect to $\frac{8}{3}g_{\mathbb{RP}^3}$. If $H$ denotes the scalar mean curvature of $\varphi$, with the averaged convention $H=\frac{1}{2}\operatorname{tr}_gA$, the mean-curvature vector of $\varphi$, considered as an immersion into $\left(\mathbb{RP}^3,\frac{8}{3}g_{\mathbb{RP}^3}\right)$, is $\widehat{\mathbf H}_{\varphi}=\frac{3}{8}HN=\sqrt{\frac{3}{8}}H\widehat N$, and hence
\begin{equation}\label{eq:scaled-mean-curvature}
|\widehat{\mathbf H}_{\varphi}|^2=\frac{3}{8}H^2.
\end{equation}

Let $\{E_1,E_2\}$ be a local $\widehat g$-orthonormal frame tangent to $\Sigma$. The composition formula for second fundamental forms gives
$$
B^X(E_i,E_j)=d\psi(B^\varphi(E_i,E_j))+B^\psi(E_i,E_j);
$$
see, for instance, \cite[Chapter~2]{Chen1973}. Taking the trace and using the minimality of $\psi$, we find
$$
\mathbf H_X=d\psi(\widehat{\mathbf H}_{\varphi})+\frac{1}{2}\sum_{i=1}^2B^\psi(E_i,E_i)=d\psi(\widehat{\mathbf H}_{\varphi})-\frac{1}{2}B^\psi(\widehat N,\widehat N\mskip1.5mu).
$$
Observe that the two terms in the last expression are orthogonal: the first belongs to $d\psi(T\mathbb{RP}^3)$, whereas the second is normal to $\psi(\mathbb{RP}^3)$ in $\mathbb{S}^8$. Therefore, it follows from \eqref{eq:veronese-normal-curvature} and \eqref{eq:scaled-mean-curvature} that
\begin{equation}\label{eq:HX}
|\mathbf H_X|^2=|\widehat{\mathbf H}_{\varphi}|^2+\frac{1}{4}|B^\psi(\widehat N,\widehat N\mskip1.5mu)|^2=\frac{3}{8}H^2+\frac{1}{8}.
\end{equation}

Since $\dim\Sigma=2$, the constant conformal change $\widehat g=\frac{8}{3}g$ also gives
\begin{equation}\label{eq:area-scaling}
d\widehat\Sigma=\frac{8}{3}d\Sigma.
\end{equation}
Combining \eqref{eq:HX} and \eqref{eq:area-scaling}, we finally obtain
\begin{equation}\label{eq:willmore-conversion}
\int_\Sigma(1+|\mathbf H_X|^2)d\widehat\Sigma
=\frac{8}{3}\int_\Sigma\Big(1+\frac{1}{8}+\frac{3}{8}H^2\Big)d\Sigma =
\int_\Sigma(3+H^2)d\Sigma.
\end{equation}

\subsection{Conformal balancing}

We shall use the following standard conformal balancing lemma, usually referred to as the Li--Yau center-of-mass argument; see \cite{ElSoufiIlias1986,LiYau}.

\begin{lemma}[Li--Yau conformal balancing]\label{lem:balance}
Let $Y:M\looparrowright\mathbb{S}^N$ be a smooth immersion of a closed manifold, and let $u\in C^\infty(M)$ be positive. Then there exists a conformal diffeomorphism $F:\mathbb{S}^N\to\mathbb{S}^N$ such that, writing $\Psi=F\circ Y=(\Psi_1,\ldots,\Psi_{N+1})$, one has
$$
\int_M u\Psi_\alpha\mskip1mu dM=0,
\qquad \alpha=1,\ldots,N+1,
$$
that is,
$$
\int_M u\Psi dM=0\quad\text{in}\quad\mathbb{R}^{N+1}.
$$
\end{lemma}

We also recall the conformal invariance of the spherical Willmore functional. Let $Y:\Sigma^2\looparrowright\mathbb{S}^N$ be an immersion of a closed surface, and let $\mathbf H_Y$ denote its mean-curvature vector in the unit sphere. For every conformal diffeomorphism $F:\mathbb{S}^N\to\mathbb{S}^N$, one has
$$
\int_\Sigma\big(1+|\mathbf H_{F\circ Y}|^2\big)d\Sigma_{F\circ Y}=\int_\Sigma\big(1+|\mathbf H_Y|^2\big)d\Sigma_Y;
$$
see \cite{LiYau}. Since
$$
\Area(F\circ Y)\leq\int_\Sigma\big(1+|\mathbf H_{F\circ Y}|^2\big)d\Sigma_{F\circ Y},
$$
it follows that
\begin{equation}\label{eq:conformal-area}
\Area(F\circ Y)\leq\int_\Sigma\big(1+|\mathbf H_Y|^2\big)d\Sigma_Y.
\end{equation}
Moreover, equality in \eqref{eq:conformal-area} holds if and only if $\mathbf H_{F\circ Y}\equiv0$, that is, if and only if $F\circ Y$ is minimal in $\mathbb{S}^N$. By a slight abuse of notation, $\Area(F\circ Y)$ denotes the area of $\Sigma$ with respect to the metric induced from $\mathbb{S}^N$ by $F\circ Y$.

\section{The second-eigenvalue estimate}\label{sec:eigenvalue-estimate}

In this section, we prove that, under the assumptions of Theorem~\ref{thm:main}, inequality \eqref{eq:quantitative-main} holds. We also derive a corollary that will be used in the analysis of the equality case in the next section.

\begin{proof}[Proof of inequality \eqref{eq:quantitative-main}]

Let $u>0$ be a first eigenfunction of $J$:
$$
Ju+\lambda_1(J)u=0.
$$
We apply Lemma~\ref{lem:balance} to $X=\psi\circ\varphi$ with weight function $u$. Thus, for some conformal diffeomorphism $F$ of $\mathbb{S}^8$, the map $\Psi=F\circ X=(\Psi_1,\ldots,\Psi_9)$ satisfies
$$
\int_\Sigma u\Psi_\alpha\mskip1mu d\Sigma=0\quad\text{for every}\quad\alpha.
$$
Each coordinate function $\Psi_\alpha$ is therefore admissible in the variational characterization of~$\lambda_2(J)$. Hence,
$$
\lambda_2(J)\int_\Sigma\Psi_\alpha^2\mskip1mu d\Sigma\le\int_\Sigma\big(|\nabla\Psi_\alpha|^2-(|\sigma|^2+2)\Psi_\alpha^2\big)d\Sigma.
$$
Summing over $\alpha$ and using $\sum_\alpha\Psi_\alpha^2=1$, we obtain
\begin{equation}\label{eq:rayleigh-sum}
\lambda_2(J)\Area(\Sigma,g)\le\int_\Sigma|d\Psi|_g^2\mskip1mu d\Sigma-\int_\Sigma(|\sigma|^2+2)d\Sigma.
\end{equation}

Since $\Psi$ is conformal and $\dim\Sigma=2$, 
$$
\int_\Sigma|d\Psi|_g^2\mskip1mu d\Sigma=2\Area(\Psi).
$$
Equations \eqref{eq:willmore-conversion} and \eqref{eq:conformal-area} then imply
$$
\Area(\Psi)\le\int_\Sigma(1+|\mathbf H_X|^2)d\widehat\Sigma=\int_\Sigma(3+H^2)d\Sigma.
$$
Inserting this into \eqref{eq:rayleigh-sum} yields
\begin{equation}\label{eq:before-gauss}
\lambda_2(J)\Area(\Sigma,g)\le2\int_\Sigma(3+H^2)d\Sigma-\int_\Sigma(|\sigma|^2+2)d\Sigma.
\end{equation}

On the other hand, the Gauss equation for $\Sigma$ in $\mathbb{RP}^3$ is 
$$
K=1+\det A.
$$
Since 
$$
|\sigma|^2=4H^2-2\det A,
$$
we have
$$
|\sigma|^2+2=4H^2+4-2K.
$$

It follows from \eqref{eq:before-gauss} that
$$
\lambda_2(J)\Area(\Sigma,g)\le2\Area(\Sigma,g)-2\int_\Sigma H^2d\Sigma
+2\int_\Sigma Kd\Sigma.
$$
The Gauss--Bonnet theorem, which is also valid for nonorientable surfaces, gives\linebreak $\int_\Sigma Kd\Sigma=2\pi\chi(\Sigma)$. Dividing both sides of the last inequality by the area proves~\eqref{eq:quantitative-main}.
\end{proof}

\begin{corollary}\label{cor:nonstrict}
Under the assumptions of Theorem~\ref{thm:main}, if $\chi(\Sigma)\le0$, then
$\lambda_2(J)\le2$. Moreover, equality holds only if
\begin{equation}\label{eq:first-equality}
H\equiv0,\qquad\chi(\Sigma)=0,
\end{equation}
and the balanced map $\Psi=F\circ X$ is minimal in $\mathbb{S}^8$.
\end{corollary}

\begin{proof}
The first assertion follows immediately from \eqref{eq:quantitative-main}. If $\lambda_2(J)=2$, both nonpositive correction terms in that inequality vanish, which gives \eqref{eq:first-equality}. Equality must also hold in \eqref{eq:conformal-area} for $Y=X$, hence $\Psi$ is minimal.
\end{proof}

\section{The equality system}\label{sec:equality-case}

From now on, we assume that $\chi(\Sigma)\le0$ and $\lambda_2(J)=2$. By Corollary~\ref{cor:nonstrict}, the original immersion $\varphi:\Sigma^2\looparrowright\mathbb{RP}^3$ is minimal and $\chi(\Sigma)=0$. In particular, by \eqref{eq:HX},
$$
|\mathbf H_X|^2=\frac{1}{8},
$$
where $X=\psi\circ\varphi:\Sigma^2\looparrowright\mathbb{S}^8$.

Every conformal diffeomorphism of the round sphere is, up to composition with an isometry of $\mathbb{S}^8$, represented by a vector in the open unit ball $B^9=\{a\in\mathbb{R}^9:|a|<1\}$. More precisely, the noncompact part of the conformal group of $\mathbb{S}^8$ is described by the family
$$
F_a:\mathbb{S}^8\to\mathbb{S}^8,\qquad a\in B^9,
$$
defined by
\begin{equation}\label{eq:Fa}
F_a(x)=\frac{(1-|a|^2)x+2(1+\langle a,x\rangle)a}{1+|a|^2+2\langle a,x\rangle}.
\end{equation}
Thus every $F\in\operatorname{Conf}(\mathbb{S}^8)$ can be written in the form
$$
F=R\circ F_a,\qquad R\in O(9),\qquad a\in B^9.
$$
This is the standard ball parametrization of the Möbius group of the sphere; see \cite{ElSoufiIlias1986,LiYau}.

For later use, set
$$
r=|a|,\qquad h_a(x)=\langle a,x\rangle,\qquad D_a(x)=1+r^2+2h_a(x).
$$

We next compute the conformal factor of $F_a$. Let $v\in T_x\mathbb{S}^8$, so that $\langle x,v\rangle=0$, and put $q=\langle a,v\rangle$. Differentiating \eqref{eq:Fa}, we obtain
$$
d(F_a)_x(v)=\frac{(1-r^2)v+2qa}{D_a(x)}-\frac{2q}{D_a(x)}F_a(x).
$$

Using $|F_a(x)|^2=1$ and the defining formula for $F_a$, a direct simplification yields
$$
\big\langle d(F_a)_x(v),d(F_a)_x(w)\big\rangle=\frac{(1-r^2)^2}{D_a(x)^2}\langle v,w\rangle
$$
for all $v,w\in T_x\mathbb{S}^8$. Hence,
$$
F_a^*g_{\mathbb{S}^8}=e^{2\rho_a}g_{\mathbb{S}^8},\qquad e^{2\rho_a(x)}=\frac{(1-r^2)^2}{(1+r^2+2\langle a,x\rangle)^2}.
$$
Equivalently,
\begin{equation}\label{eq:rho-log}
\rho_a(x)=\log(1-r^2)-\log(1+r^2+2\langle a,x\rangle).
\end{equation}

Applying this to the immersion $X:\Sigma^2\looparrowright\mathbb{S}^8$, let
$$
h=\langle a,X\rangle,\qquad D=1+r^2+2h,\qquad e^{2\rho}=\frac{(1-r^2)^2}{D^2}.
$$
Then, if $\widehat g=X^*g_{\mathbb{S}^8}$, the metric induced on $\Sigma$ by $F_a\circ X$ is
$$
g_{F_a\circ X}=(F_a\circ X)^*g_{\mathbb{S}^8}=e^{2\rho}\widehat g.
$$
Since $\dim\Sigma=2$, the corresponding area element satisfies
$$
d\Sigma_{F_a\circ X}=e^{2\rho}d\widehat\Sigma=\frac{(1-r^2)^2}{D^2}d\widehat\Sigma.
$$

\subsection{The conformal factor and curvature}

Equality in the sum of the coordinate Rayleigh inequalities implies that equality holds in each individual inequality. Indeed, each of the corresponding nine Rayleigh defects is nonnegative, and their sum vanishes. Hence,
$$
J\Psi_\alpha+2\Psi_\alpha=0,\qquad \alpha=1,\ldots,9.
$$
Equivalently, regarding $\Psi=(\Psi_1,\ldots,\Psi_9)$ as an $\mathbb{R}^9$-valued map,
\begin{equation}\label{eq:vector-eigen}
J\Psi+2\Psi=0.
\end{equation}

By the equality statement in the conformal-area estimate, $\Psi=F_a\circ X$ is a minimal immersion into the unit sphere. Moreover, as
$$
\Psi^*g_{\mathbb{S}^8}=e^{2\rho}\widehat g=\frac{8}{3}e^{2\rho}g,
$$
the immersion $\Psi:(\Sigma^2,g)\looparrowright\mathbb{S}^8$ is conformal. Since every conformal minimal immersion of a surface into a Riemannian manifold is harmonic, it follows that, regarding $\Psi$ as an $\mathbb{R}^9$-valued map, we have
\begin{equation}\label{eq:harmonic-Psi}
\Delta_g\Psi+|d\Psi|_g^2\mskip1mu\Psi=0.
\end{equation}
Taking the trace with respect to $g$ in the pullback-metric identity yields
\begin{equation}\label{eq:energy-Psi}
|d\Psi|_g^2=\operatorname{tr}_g(\Psi^*g_{\mathbb{S}^8})=2\left(\frac{8}{3}e^{2\rho}\right)
=\frac{16}{3}e^{2\rho}.
\end{equation}

On the other hand, using 
$$
J=\Delta_g+|\sigma|^2+2
$$
into \eqref{eq:vector-eigen}, we obtain
\begin{equation}\label{eq:Psi-potential}
\Delta_g\Psi+(|\sigma|^2+4)\Psi=0.
\end{equation}
Since the original immersion $\varphi:\Sigma^2\looparrowright\mathbb{RP}^3$ is minimal, the Gauss equation in $\mathbb{RP}^3$, whose sectional curvature is one, gives $K=1-\frac{1}{2}|\sigma|^2$, that is, $|\sigma|^2=2-2K$. Hence \eqref{eq:Psi-potential} becomes
\begin{equation}\label{eq:Psi-curvature}
\Delta_g\Psi+(6-2K)\Psi=0.
\end{equation}

Comparing \eqref{eq:harmonic-Psi}, \eqref{eq:energy-Psi}, and \eqref{eq:Psi-curvature}, and using $|\Psi|^2=1$, we find
$$
\frac{16}{3}e^{2\rho}=6-2K.
$$
Therefore,
\begin{equation}\label{eq:rho-K}
e^{2\rho}=\frac{3}{8}(3-K).
\end{equation}

Since $\widehat g=\frac{8}{3}g$, the Gaussian curvature transforms under this constant rescaling according to $\widehat K=\frac{3}{8}K$. Substituting this into \eqref{eq:rho-K}, we obtain
\begin{equation}\label{eq:rho-curvature}
e^{2\rho}=\frac{3}{8}(3-K)=\frac{9}{8}-\widehat K.
\end{equation}

Finally, the explicit conformal factor of $F_a$ is
$$
e^{2\rho}=\frac{(1-r^2)^2}{(1+r^2+2h)^2},\qquad r=|a|,\qquad h=\langle a,X\rangle.
$$
Combining this identity with \eqref{eq:rho-curvature} yields
\begin{equation}\label{eq:K-h}
\widehat K=\frac{9}{8}-\frac{(1-r^2)^2}{(1+r^2+2h)^2}.
\end{equation}

\subsection{Minimality after the Möbius transformation}\label{subsec:Mobius}

By equality in the conformal-area estimate, $\Psi=F_a\circ X$ is minimal in $\mathbb{S}^8$. We now use the transformation law for the mean-curvature vector under an ambient
conformal change. If $F_a^*g_{\mathbb{S}^8}=e^{2\rho}g_{\mathbb{S}^8}$, then
$$
\mathbf H_{F_a\circ X}=e^{-2\rho}dF_a\big(\mathbf H_X-(\nabla^{\mathbb{S}^8}\rho)^\perp\big).
$$
Since $\mathbf H_{F_a\circ X}=0$, it follows that $\mathbf H_X=(\nabla^{\mathbb{S}^8}\rho)^\perp$. Here and below, the normal projection is taken with respect to $X(\Sigma)\subset\mathbb{S}^8$.

From \eqref{eq:rho-log}, restricted along $X$, we have
$$
\rho=\log(1-r^2)-\log D,\qquad D=1+r^2+2h,\qquad h=\langle a,X\rangle.
$$
Since $\nabla^{\mathbb{S}^8}h=a-hX$, we obtain $\nabla^{\mathbb{S}^8}\rho=-\frac{2}{D}(a-hX)$. So, because $X(p)$ is radial and therefore orthogonal to $T_{X(p)}\mathbb{S}^8$ for each $p\in\Sigma$, this gives
\begin{equation}\label{eq:HX-aperp}
\mathbf H_X=-\frac{2}{D}a^\perp.
\end{equation}

The original immersion $\varphi:\Sigma^2\looparrowright\mathbb{RP}^3$ is minimal. Thus, by \eqref{eq:HX}, $|\mathbf H_X|^2=\frac{1}{8}$. Combining this identity with \eqref{eq:HX-aperp}, we find 
\begin{equation}\label{eq:aperp}
|a^\perp|^2=\frac{D^2}{32}.
\end{equation}

On the other hand, the orthogonal decomposition of the constant vector $a\in\mathbb{R}^9$ along $X(\Sigma)\subset\mathbb{S}^8$ is
$$
a=hX+\widehat\nabla h+a^\perp.
$$
Thus,
$$
r^2=h^2+|\widehat\nabla h|^2+|a^\perp|^2.
$$
Using \eqref{eq:aperp} and $D=1+r^2+2h$, we obtain
\begin{align}
|\widehat\nabla h|^2&=r^2-\frac{D^2}{32}-h^2\notag\\
&=r^2-\frac{(1+r^2)^2}{32}-\frac{1+r^2}{8}h-\frac{9}{8}h^2\notag\\
&=r^2-\frac{(1+r^2)^2}{36}-\frac{9}{8}\left(h+\frac{1+r^2}{18}\right)^2.
\label{eq:transnormal-h}
\end{align}

For an immersed surface in the unit sphere, with the averaged convention
for the mean-curvature vector,
$$
\widehat\Delta X=-2X+2\mathbf H_X.
$$
Thus, taking the Euclidean scalar product with the constant vector $a$ and
using \eqref{eq:HX-aperp} and \eqref{eq:aperp}, gives
\begin{align}
\widehat\Delta h&=-2h+2\langle a,\mathbf H_X\rangle=-2h-\frac{4}{D}|a^\perp|^2\notag\\
&=-2h-\frac{D}{8}=-\frac{9}{4}h-\frac{1+r^2}{8}\notag\\
&=-\frac{9}{4}\left(h+\frac{1+r^2}{18}\right).
\label{eq:lap-h}
\end{align}

Defining
\begin{equation}\label{eq:def-f}
f=h+\frac{1+r^2}{18}
\end{equation}
and substituting this into \eqref{eq:transnormal-h} and \eqref{eq:lap-h} yields
\begin{equation}\label{eq:isoparametric}
|\widehat\nabla f|^2=C_r-\frac{9}{8}f^2,\qquad\widehat\Delta f=-\frac{9}{4}f,
\end{equation}
where 
$$
C_r=r^2-\frac{(1+r^2)^2}{36}.
$$

\begin{proposition}\label{prop:obata-equation}
The function $f$ defined by \eqref{eq:def-f} satisfies
\begin{equation}\label{eq:obata}
\widehat\Hess f=-\frac{9}{8}f\mskip1mu\widehat g\quad\mbox{on}\quad\Sigma.
\end{equation}
\end{proposition}

\begin{proof}
If $f$ is constant, the second equation in \eqref{eq:isoparametric} gives $f\equiv0$, and \eqref{eq:obata} follows immediately.

Assume that $f$ is nonconstant and let $\mathcal R_f=\{p\in\Sigma:\widehat\nabla f(p)\neq0\}$ be the regular set of $f$. On $\mathcal R_f$, choose a local $\widehat g$-orthonormal frame $\{e_1,e_2\}$ such that $e_1=\frac{\widehat\nabla f}{|\widehat\nabla f|}$. Differentiating the transnormal identity in \eqref{eq:isoparametric} gives
$$
2\mskip1mu\widehat\Hess f(\mskip1mu\widehat\nabla f,\mskip2mu\cdot\mskip2mu)=-\frac{9}{4}f\mskip.5mu df.
$$
Therefore, it follows that
$$
\widehat\Hess f(e_1,e_1)=-\frac{9}{8}f,\qquad\widehat\Hess f(e_1,e_2)=0.
$$
Thus, taking the trace of $\widehat\Hess f$ and using $\widehat\Delta f=-\frac{9}{4}f$, we obtain
$$
\widehat\Hess f(e_2,e_2)=\widehat\Delta f-\widehat\Hess f(e_1,e_1)=-\frac{9}{8}f.
$$
Hence \eqref{eq:obata} holds on $\mathcal R_f$.

We claim that $\mathcal R_f$ is dense in $\Sigma$. Indeed, if $\widehat\nabla f=0$ on a nonempty open set $U\subset\Sigma$, then $f$ is constant on each connected component of $U$. Since $\widehat\Delta f=0$ on $U$, the equation $\widehat\Delta f+\frac{9}{4}f=0$ implies that $f=0$ on $U$. The unique continuation property for Laplace eigenfunctions then yields $f\equiv0$ on $\Sigma$, contradicting the assumption that $f$ is nonconstant.

Finally, \eqref{eq:obata} extends from $\mathcal R_f$ to all of $\Sigma$ by continuity.
\end{proof}

\section{Exclusion of equality}\label{sec:flat-minimal}

\subsection{Conclusion of the equality case}

We now distinguish whether the function $f$ defined in \eqref{eq:def-f} is constant or not.

Suppose first that $f$ is nonconstant. By Proposition~\ref{prop:obata-equation} and Obata's theorem \cite{Obata}, the closed Riemannian surface $(\Sigma,\widehat g)$ is isometric to the round sphere of Gaussian curvature $\widehat K\equiv\frac{9}{8}$, which is incompatible with \eqref{eq:rho-curvature}: 
$$
e^{2\rho}=\frac{9}{8}-\widehat K\equiv0.
$$
Hence $f$ must be constant.

The second equation in \eqref{eq:isoparametric} then yields $f\equiv0$. So, by \eqref{eq:def-f}, the height function $h$ is constant $h\equiv-\frac{1+r^2}{18}$. It follows from \eqref{eq:K-h} that $\widehat K$ is constant. Since $\chi(\Sigma)=0$, the Gauss--Bonnet theorem gives 
$$
0=2\pi\chi(\Sigma)=\int_\Sigma\widehat Kd\widehat\Sigma=\widehat K\Area(\Sigma,\widehat g).
$$
Therefore $\widehat K\equiv0$. Since $\widehat g=\frac{8}{3}g$, we have $K=\frac{8}{3}\widehat K\equiv0$. Thus $\varphi:\Sigma^2\looparrowright\mathbb{RP}^3$ is a closed flat minimal immersion.

\begin{lemma}\label{lem:flat-minimal}
Let $\varphi:\Sigma^2\looparrowright\mathbb{RP}^3$ be a closed flat minimal immersion. Then:
\begin{enumerate}[label=\textup{(\roman*)}]
\item $\Sigma$ is orientable;
\item $\varphi$ is a finite covering of the projective Clifford torus;
\item for $J=\Delta+|\sigma|^2+2$, one has $\lambda_2(J)\leq0$. 
\end{enumerate}
\end{lemma}

\begin{proof}
Let $\pi:\mathbb{S}^3\to\mathbb{RP}^3$ be the antipodal covering, and let $p:\widetilde\Sigma\to\Sigma$ denote the universal covering of $\Sigma$. Since $\widetilde\Sigma$ is simply connected, the map $\varphi\circ p:\widetilde\Sigma\to\mathbb{RP}^3$ admits a lift
$$
\widetilde\varphi:\widetilde\Sigma\to\mathbb{S}^3
$$
such that 
$$
\pi\circ\widetilde\varphi=\varphi\circ p.
$$

Because $\pi$ is a local isometry, the lifted immersion $\widetilde\varphi$ is also flat and minimal. Let $\kappa_1,\kappa_2$ denote its principal curvatures. The minimality of $\widetilde\varphi$ and the Gauss equation give that
$$
\kappa_1+\kappa_2=2\widetilde H=0,\qquad 1+\kappa_1\kappa_2=\widetilde K=0.
$$
Hence, after choosing the unit normal appropriately, $\kappa_1=1$ and $\kappa_2=-1$. In particular, the principal curvatures are distinct and constant. The Codazzi equations imply that the connection form of a local principal orthonormal frame vanishes. Thus the principal directions are parallel. Equivalently, by Lawson's local classification \cite[Corollary~3]{Lawson1969}, the lifted immersion is locally congruent to an open subset of the Clifford torus.

Since $\Sigma$ is compact and flat, its universal Riemannian covering is complete, simply connected, and flat, and hence isometric to $\mathbb{R}^2$. The parallel principal frame therefore extends globally, and integration of the corresponding structure equations shows that, after an ambient isometry,
\begin{equation}\label{eq:universal-Clifford}
\widetilde\varphi(s,t)=\frac{1}{\sqrt{2}}\left(e^{is},e^{it}\right),\qquad(s,t)\in\mathbb{R}^2.
\end{equation}
So, the image of $\widetilde\varphi$ is the Clifford torus
$$
\mathbb{T}_{\mathrm{Cl}}=\left\{\frac{1}{\sqrt{2}}\left(e^{is},e^{it}\right):s,t\in\mathbb{R}\right\}\subset\mathbb{S}^3.
$$
The local classification invoked above is classical; see \cite[Corollary~3]{Lawson1969}; see also \cite{Lawson1970}.

The antipodal map preserves $\mathbb{T}_{\mathrm{Cl}}$, since $-\widetilde\varphi(s,t)=\widetilde\varphi(s+\pi,t+\pi)$. Therefore, its quotient
$$
\overline{\mathbb{T}}_{\mathrm{Cl}}=\pi\left(\mathbb{T}_{\mathrm{Cl}}\right)\subset\mathbb{RP}^3
$$
is a flat embedded torus, referred to as the projective Clifford torus. Since the image of $\varphi$ is contained in $\overline{\mathbb{T}}_{\mathrm{Cl}}$, the original immersion may be regarded as a local isometry $\varphi:\Sigma\to\overline{\mathbb{T}}_{\mathrm{Cl}}$. Its image is both open, because $\varphi$ is a local diffeomorphism, and compact, hence closed. Since $\overline{\mathbb{T}}_{\mathrm{Cl}}$ is connected, $\varphi$ is surjective. A surjective local isometry from the compact surface $\Sigma$ is a finite covering. Thus $\Sigma$ is a finite covering of a torus and is therefore itself a torus; in particular, it is orientable. This proves \textup{(i)} and \textup{(ii)}.

It remains to prove \textup{(iii)}. Since $K=H=0$, the Gauss equation gives $|\sigma|^2=2$. Therefore, 
$$
J=\Delta+4.
$$

We first compute the first positive Laplace eigenvalue of the projective Clifford torus. In the coordinate system \eqref{eq:universal-Clifford}, the induced metric is
$$
g_{\mathrm{Cl}}=\frac{1}{2}\left(ds^2+dt^2\right).
$$
Smooth functions on the projective quotient correspond to functions on $\mathbb{R}^2$ invariant under the translations
$$
(s,t)\longmapsto(s+2\pi,t),\qquad(s,t)\longmapsto(s,t+2\pi),\qquad(s,t)\longmapsto(s+\pi,t+\pi).
$$
Accordingly, the Fourier mode $e^{i(ms+nt)}$, $m,n\in\mathbb{Z}$, descends to the projective Clifford torus if and only if $e^{i\pi(m+n)}=1$, that is, $m+n\equiv0\pmod2$. Furthermore, it corresponds to the Laplace eigenvalue $\mu(m,n)=2(m^2+n^2)$. The smallest positive value of $\mu(m,n)$ for which $(m,n)$ satisfies the parity condition is attained for $(m,n)=(\pm1,\pm1)$, and is equal to $4$. Therefore $\mu_1(\mskip1mu\overline{\mathbb{T}}_{\mathrm{Cl}})=4$.

Since $\varphi:\Sigma\to\overline{\mathbb{T}}_{\mathrm{Cl}}$ is a finite Riemannian covering, the pullback of any Laplace eigenfunction on $\overline{\mathbb{T}}_{\mathrm{Cl}}$ with eigenvalue $4$ is a nonconstant eigenfunction on $\Sigma$ corresponding to the same eigenvalue. So, $\mu_1(\Sigma)\leq4$.

Thus the constant functions give the first eigenvalue $\lambda_1(J)=-4$, whereas the second eigenvalue is $\lambda_2(J)=\mu_1(\Sigma)-4\leq0$.
\end{proof}

\begin{proof}[Proof of Theorem~\ref{thm:main}]
In Section~\ref{sec:eigenvalue-estimate}, we prove inequality \eqref{eq:quantitative-main}, and hence $\lambda_2(J)\le2$ when $\chi(\Sigma)\le0$. Suppose that equality holds. The analysis in Section~\ref{sec:equality-case} produces a function $f$ satisfying \eqref{eq:obata}. If $f$ is nonconstant, Obata's theorem contradicts \eqref{eq:rho-curvature}. If $f$ is constant, the surface is flat and minimal, and Lemma~\ref{lem:flat-minimal} gives $\lambda_2(J)\le0$, again a contradiction. Therefore $\lambda_2(J)<2$.
\end{proof}

\begin{remark}
The one-sided case is already excluded in the constant branch of the argument: a closed nonorientable surface with Euler characteristic zero is a Klein bottle, whereas Lemma~\ref{lem:flat-minimal} shows that every closed flat minimal surface immersed in $\mathbb{RP}^3$ is orientable. Thus the proof does not require a separate spectral problem on a double cover.
\end{remark}

\section{Final comments}\label{sec:final-comments}

Observe that \eqref{eq:before-gauss} can be rewritten as
$$
\lambda_2(J)\Area(\Sigma,g)\le4\Area(\Sigma,g)-\int_\Sigma(|\sigma|^2-2H^2)d\Sigma.
$$
Since $|\sigma|^2\ge2H^2$, it follows that
$$
\lambda_2(J)\le4
$$
for every closed connected immersion $\varphi:\Sigma^2\looparrowright\mathbb{RP}^3$.

Very recently, the authors \cite{BatistaMendes} characterized the totally geodesic copies of $\mathbb{RP}^2$ in $\mathbb{RP}^3$ as the only closed immersed surfaces satisfying $\lambda_2(J)=4$. More generally, they established a sharp upper bound for the second eigenvalue of the Schrödinger operator $\Delta+|\sigma|^2+k$ on closed immersed $k$-dimensional submanifolds of the compact projective spaces $\mathbb{FP}^m$, $\mathbb{F}\in\{\mathbb{R},\mathbb{C},\mathbb{H}\}$, as well as of the Cayley projective plane $\mathbb{OP}^2$. They also showed that equality can occur only for totally umbilical submanifolds.

In the real projective case, they showed that if $\iota:\Sigma^k\to\mathbb{RP}^m$ is a closed connected immersion, then
$$
\lambda_2(\Delta+|\sigma|^2+k)\le k+2.
$$
Moreover, equality holds if and only if $\iota$ is an embedding onto a totally geodesic copy of $\mathbb{RP}^k$ in $\mathbb{RP}^m$; see \cite[Corollary~4.4]{BatistaMendes}.

For $k=2$, this shows that inequality \eqref{eq:quantitative-main},
$$
\lambda_2(\Delta+|\sigma|^2+2)\le2-\frac{2}{\Area(\Sigma,g)}\int_\Sigma H^2d\Sigma+\frac{4\pi\chi(\Sigma)}{\Area(\Sigma,g)},
$$
is attained by the totally geodesic embedding $\mathbb{RP}^2\subset\mathbb{RP}^3$. Indeed, since $\mathbb{RP}^2$ is minimal, $\chi(\mathbb{RP}^2)=1$, and $\Area(\mathbb{RP}^2)=2\pi$, the right-hand side of \eqref{eq:quantitative-main} equals $4$.

\section*{Funding}
The authors were partially supported by the Brazilian National Council for Scientific and Technological Development, Brazil [Grants: 402563/2023-9 and 304381/2026-8 to M.B.; 309867/2023-1 and 445723/2025-4 to A.M.], and were partially supported by Coordination for the Improvement of Higher Education Personnel [Finance code - 001].

\bibliographystyle{amsplain}
\bibliography{bibliography.bib}

\end{document}